\documentclass[reqno,11pt]{amsart}
\usepackage{aliascnt}
\usepackage{hyperref}
\usepackage{amsfonts}
\usepackage{mathrsfs}
\usepackage{amsmath}
\usepackage{amsmath,graphics}%diagrams,
\usepackage{amssymb}
\usepackage{indentfirst,latexsym,bm,amsmath,pstricks,amssymb,amsthm,graphicx,fancyhdr,float,color}
\usepackage[all]{xy}
\usepackage{amsthm}
\usepackage{amscd}
\usepackage{hyperref}
\usepackage[all]{hypcap}
\newcommand{\FF}{\mathbb{F}}

\newcommand{\NN}{\mathbb{N}}

\newcommand{\fg}{\mathfrak{g}}

\newcommand{\fh}{\mathfrak{h}}

\newcommand{\fn}{\mathfrak{n}}

\newcommand{\ft}{\mathfrak{t}}

\DeclareMathOperator{\Aut}{Aut}

\DeclareMathOperator{\Ext}{Ext}
\DeclareMathOperator{\GL}{GL}

\DeclareMathOperator{\rk}{rk}

\DeclareMathOperator{\Tor}{Tor}

\theoremstyle{plain}
\numberwithin{equation}{section}
\newtheorem{Theorem}{Theorem}[section]
\newtheorem{Lemma}[Theorem]{Lemma}
\newtheorem{Corollary}[Theorem]{Corollary}
\newtheorem{Proposition}[Theorem]{Proposition}

\theoremstyle{Theorem}

\theoremstyle{remark}
\newtheorem*{Remark}{Remark}
\newtheorem*{Remarks}{Remarks}

\newtheorem*{Example}{Example}

\numberwithin{equation}{section}

\begin{document}
\title{A note on the rank of a restricted Lie algebra}
\author{Hao Chang}
\address{Mathematisches Seminar, Christian-Albrechts-Universit\"at zu Kiel, Ludewig-Meyn-Str. 4, 24098 Kiel, Germany}
\email{chang@math.uni-kiel.de}
\date{\today}
\begin{abstract}
In this short note, we study the rank of a restricted Lie algebra $(\fg,[p])$ and give some applications, which concerns the dimensions of non-trivial irreducible modules.
We also compute the rank of the restricted contact algebra $K(n)$.
\end{abstract}
\maketitle

\section{Introduction}
The theorem on the conjugacy of Cartan subalgebras in finite-dimensional Lie algebras over an algebraically closed field of characteristic $0$ is a classical result in the Lie theory.
However, it can not be extended to Lie algebras of prime characteristic.

Let $k$ be an algebraically closed field
of characteristic $\mathrm{char}(k)=p>0$ and $(\fg,[p])$
be a restricted Lie algebra. By general theory, the Cartan subalgebras of $\fg$ are precisely the centralizers of maximal tori, see for instance \cite[\S 2, Theorem 4.1]{SF}. 
There exist numerous examples of restricted Lie algebras that contain Cartan subalgebras of different dimensions.
We are interested in the tori of maximal dimension, which play
an important role in the classification theory (\cite{BW}).
Premet's theorem ensures that the invariance of the dimension of Cartan subalgebras with maximal toral part (\cite[Theorem 1]{Pre}, see also \cite[(3.5(1))]{Fa04}). The dimension is the so-called rank of $\fg$, say $\rk(\fg)$.

The aim of this note is to study the rank of a restricted Lie algebra.
We give some applications, which concerns the dimensions of non-trivial irreducible modules. We also compute the rank of the restricted contact algebra $K(n)$.

\section{The rank of a Lie algebra}
Throughout, we shall be working over an algebraically closed field $k$
of characteristic $\mathrm{char}(k)=p>0$.
Given a restricted Lie algebra $(\fg,[p])$,
we denote by $\mu(\fg)$ and $\rk(\fg)$ the maximal dimension of all tori
$\ft\subseteq\fg$ and the minimal dimension of all Cartan subalgebra $\fh\subseteq\fg$, respectively. 
Moreover, $\dim_k C_{\fg}(\ft)=\rk(\fg)$ for every torus $\ft$ of maximal dimension (\cite[Theorem 1]{Pre}).
The following Lemma is well-known, see for example \cite[Theorem 2.16]{Winter} for maximal tori and \cite[Lemma 1.7.2]{BW} for tori of maximal
dimension.

\begin{Lemma}\label{mu}
Let $(\fg,[p])$ be a restricted Lie algebra,
$\fn\unlhd\fg$ be a $p$-ideal.
Then the following statements hold:
\begin{enumerate}
\item[(1)] If $\ft\subseteq\fg$ is a torus of maximal dimension,
then $\ft\cap\fn$ and $\ft+\fn/\fn$ are tori of maximal dimension of $\fn$
and $\fg/\fn$, respectively.
\item[(2)] $\mu(\fg)=\mu(\fn)+\mu(\fg/\fn)$.
\end{enumerate}
\end{Lemma}

In fact, we can find some subtorus $\ft'\subseteq\ft$ such that
$\ft=(\ft\cap\fn)\oplus\ft'$ (see for example \cite[Page 1347]{BFS}).

\begin{Proposition}\label{rk}
Let $\fn\unlhd\fg$ be a $p$-ideal. Then the following statements hold:
\begin{enumerate}
\item[(1)] $\rk(\fg)\leq\rk(\fn)+\rk(\fg/\fn)$.
\item[(2)] If $\ft\subseteq\fg$ is a torus of maximal dimension, then
$$\rk(\fg)=\rk(\ft+\fn)+\rk(\fg/\fn)-\mu(\fg/\fn).$$
\end{enumerate}
\end{Proposition}
\begin{proof}
Let $\ft\subseteq\fg$ be a torus of maximal dimension.
We write $\ft=(\ft\cap\fn)\oplus\ft'$ and observe that the short exact sequence
$$(0)\rightarrow\fn\rightarrow\fg\rightarrow\fg/\fn\rightarrow(0)$$
of $\ft$-modules splits. Consequently, $C_{\fg}(\ft)\cong C_{\fn}(\ft)\oplus C_{\fg/\fn}(\ft)$.
Since $C_{\fn}(\ft)\subseteq C_{\fn}(\ft\cap\fn)$, Lemma \ref{mu} yields
\begin{equation}\label{1}
\rk(\fg)=\dim_k C_{\fn}(\ft)+\rk(\fg/\fn)\leq\rk(\fn)+\rk(\fg/\fn).
\end{equation}

For the proof of (2), we note that
$$\rk(\ft+\fn)=\dim_k C_{\ft+\fn}(\ft)=\dim_k C_{\fn}(\ft)+\dim_k\ft'=\dim_k C_{\fn}(\ft)+\mu(\fg/\fn),$$
so that (\ref{1}) implies
\begin{equation}
\rk(\fg)=\dim_kC_{\fn}(\ft)+\rk(\fg/\fn)=\rk(\ft+\fn)+\rk(\fg/\fn)-\mu(\fg/\fn),
\end{equation}
as desired.
\end{proof}

A restricted Lie algebra $(\fg,[p])$ is called $p$-nilpotent if $\fg$ consists
of only $p$-nilpotent elements. In particular, $\fg$ is $p$-nilpotent
if and only if $\mu(\fg)=0$.

\begin{Corollary}\label{corollary}
Let $\fn\unlhd\fg$ be a p-ideal. If $\fg/\fn$ is $p$-nilpotent, then
$\mu(\fg)=\mu(\fn)$ and $\rk(\fg)=\rk(\fn)+\dim_k\fg/\fn$.
\end{Corollary}
\begin{proof}
By assumption, we know $\mu(\fg/\fn)=0$ and $\dim_k\fg/\fn=\rk(\fg/\fn)$,
so that our assertions follow from
Lemma \ref{mu}(2) and Proposition \ref{rk}(2).
\end{proof}

Let $(\fg,[p])$ be a restricted Lie algebra.
We denote by $$\Tor(\fg):=\{\ft\subseteq\fg;~\dim_k\ft=\mu(\fg)\}$$ the
toral variety of $\fg$ (see \cite{Fa04}, \cite{BFS}).
Moreover, the variety $\Tor(\fg)$ is irreducible (\cite[(1.6)]{FV01}).
Let $\ft\in\Tor(\fg)$ be a torus of maximal dimension.
We denote by $S(\fg,\ft)$ the toral stabilizer of $\fg$ relative to $\ft$ (see \cite[Page 4194]{Fa04} for details).
In particular, $S(\fg,\ft)$ is a subgroup of $\Aut_p(\ft)\cong\GL_{\mu(\fg)}(\FF_p)$.

%Let $(\fg,[p])$ be a restricted Lie algerba. 
%A $p$-ideal $\fn\unlhd\fg$ is called elementary abelian if $[\fn,\fn]=\{0\}=\fn^{[p]}$. Furthermore, Let $V$ be a restricted $\fg$-module.
%We construct a new restricted Lie algebra $\widetilde{\fg}:=\fg\ltimes V$,
%whose bracket and $p$-map are given by
%\begin{equation}\label{new lie algebra}
%[(x_1,v_1),(x_2,v_2)]=([x_1,x_2],x_1.v_2-x_2.v_1);~~(x,v)^{[p]}=(x^{[p]},x^{p-1}.v)
%\end{equation}
%for any $x_1,x_2,x\in\fg$ and $v_1,v_2,v\in V$.
%It is easy to check that $V$ is an elementary abelian ideal of $\widetilde{\fg}$.

\begin{Theorem}\label{trivial modules}
Let $V$ be a restricted $\fg$-module of Lie algebra $(\fg,[p])$, and $\ft\in\Tor(\fg)$ a torus of maximal dimension. Suppose that $S(\fg,\ft)\cong\GL_{\mu(\fg)}(\FF_p)$. 
\begin{enumerate}
\item[(i)] If $\dim_k V<p^{\mu(\fg)}-1$, then the trivial $\fg$-module $k$ is the only composition factor of $V$ (i.e. $V=k^{\dim_k V}$ in the Grothendieck group).
\item[(ii)] Moreover, if $\fg=\fg^{[p]}+[\fg,\fg]$, then $V\cong k^{\dim_k V}$. 
\end{enumerate}
\end{Theorem}
\begin{proof}
Let
\begin{equation}\label{weight space decomp}
V=\bigoplus\limits_{\lambda\in\Lambda_V}V_{\lambda}
\end{equation}
relative to $\ft$.
According to \cite[Corollary 6.2(2)]{BFS}, we know that all weight spaces
belonging to a non-zero weights have the same dimension.
Let
$$I:=\{x\in\fg;~x.V=(0)\}$$
be the annihilator of $V$ in $\fg$.
By assumption, this is a $p$-ideal that contains $\ft$,
It follows from Lemma \ref{mu}(2) that $\mu(\fg/I)=0$,
so that $\fg/I$ is $p$-nilpotent. We obtain (i).

(ii) A classical result by Hochschild (see \cite[Theorem 2.1]{Hochschild1}) states that 
\begin{equation}
   \Ext^1_{U_0(\fg)}(k,k)\cong H^1_\ast(\fg,k) \cong (\fg/(\fg^{[p]}+[\fg,\fg]))^\ast =(0).
\end{equation}
Now (i) implies the assertion.
\end{proof}

\begin{Corollary}\label{subalgebra rank}
Let $\fh\subseteq\fg$ be a $p$-subalgebra of the restricted Lie algebra $(\fg,[p])$. Suppose that $\mu(\fh)=\mu(\fg)$.
If $\ft\in\Tor(\fh)$ such that
\begin{enumerate}
\item[(a)] $S(\fh,\ft)\cong\GL_{\mu(\fh)}(\FF_p)$, and
\item[(b)] $\dim_k\fg-\dim_k\fh<p^{\mu(\fg)}-1$,
\end{enumerate}
then $\rk(\fg)=\rk(\fh)+\dim_k\fg-\dim_k\fh$.
\end{Corollary}
\begin{proof}
Thanks to Theorem \ref{trivial modules},
the torus $\ft$ of $\fh$ acts trivially on $\fg/\fh$.
Hence, $\fg_\alpha=\fh_\alpha$ for every root $\alpha$ and 
$$\dim_k\fg-\dim_k\fh=\dim_kC_{\fg}(\ft)-\dim_kC_{\fh}(\ft)=\rk(\fg)-\rk(\fh).$$
\end{proof}

\begin{Corollary}\label{ideal}
Let $\fh\subseteq\fg$ be a simple $p$-subalgebra of the restricted Lie algebra $(\fg,[p])$.
If $\ft\in\Tor(\fh)$ such that
\begin{enumerate}
\item[(a)] $S(\fh,\ft)\cong\GL_{\mu(\fh)}(\FF_p)$, and
\item[(b)] $\dim_k\fg-\dim_k\fh<p^{\mu(\fh)}-1$,
\end{enumerate}
then $\fh$ is an ideal in $\fg$.
\end{Corollary}
\begin{proof}
Now apply Theorem \ref{trivial modules} to the $\fh$-module $\fg/\fh$.
Note that $\fh$ is simple, it follows that the module $\fg/\fh$ is trivial.
\end{proof}

\begin{Example}\label{example W(n) summand}
Let $W(n)$ be the restricted Jacobson-Witt algebra.
Suppose that $W(n) \subseteq\fg$ has codimension $<p^n-1$.
Note that $W(n)$ is simple, hence is an ideal of $\fg$ (Corollary \ref{ideal}).
We have the short exact sequences of restricted Lie algebras:
$$(0)\rightarrow W(n)\rightarrow\fg\rightarrow\fg/W(n)\rightarrow(0).$$
Note that $W(n)$ is centerless with all derivations being inner (see \cite[\S 4,  Theorem 8.5]{SF}),
it follows from a direct computation that
$$\fg=C_{\fg}(W(n))\oplus W(n).$$ 
\end{Example}

\begin{Corollary}\label{rep of simple Lie algebras}
Suppose that $p>3$.
Let $\fg$ be a non-classical restricted simple Lie algebra and $V$ be a restricted $\fg$-module. If $\dim_kV<p^{\mu(\fg)}-1$,
then $V$ is a trivial module.
\end{Corollary}
\begin{proof}
According to the classification of restricted simple Lie algebras (cf. \cite{S1}), we know that $\fg$ should be a Lie algebra of Cartan type or Melikian algebra. Thanks to \cite[Theorem 5.3, 5.5]{BFS}, we have
$$S(\fg,\ft)\cong\GL_{\mu(\fg)}(\FF_p).$$
Since $\fg$ is simple, now Theorem \ref{trivial modules} implies that $V$ must be trivial. 
\end{proof}
\begin{Remark}
In \cite[Lemma 3.6]{Jantzen}, Jantzen gave a computation of the lower bound of non-trivial irreducible modules for the restricted simple Lie algebras of Cartan type.
\end{Remark}

\begin{Example}
The simple modules for the Witt algebra $W(1)$ were determined in \cite{Chang}. There are precisely $p$ isomorphism classes of simple
restricted $W(1)$-modules.
These modules are represented by $L(i);~0\leq i\leq p-1$ and the dimensions are given by $\dim_kL(0)=1$, $\dim_kL(p-1)=p-1$ and $\dim_kL(i)=p;~1\leq i\leq p-2$.

More generally, if a simple Lie algebra $S$ possesses a faithful irreducible representation of dimension $\leq p-1$, then $S$ is either classical or the Witt algebra $W(1)$ (see \cite{S0}).
\end{Example}

%%%%%%%%%%%%%%%%%%%%%%%%%%%%%%%%%%%%%%%%%%%%%%%%%%%%%%%%%%%%%%%%%%%
\section{Rank of Contact algebra $K(n)$}
In this section, we assume that $p\geq 3$.
The reader is referred to \cite[\S 4]{SF} for basic facts concerning the Lie algebras of Cartan type.
We denote by $W(n)$, $S(n)$, $H(n);~n=2r$ and $K(n);~n=2r+1$ the restricted simple Lie algebras of Cartan type.
We recall the following results due to Demushkin \cite{D1}, \cite{D}, as corrected in \cite[\S7.5]{S1}.

\begin{Lemma}\label{D's result}
\begin{enumerate}
\item[(1)] $\mu(W(n))=\rk(W(n))=n$.
\item[(2)] $\mu(S(n))=n-1$ \rm{and} $\rk(S(n))=(n-1)(p-1)$.
\item[(3)] $\mu(H(n))=r$ \rm{and} $\rk(H(n))=p^r-2$.
\item[(4)] $\mu(K(n))=r+1$.
\end{enumerate}
\end{Lemma}
In this section, we would like to compute the rank of $K(n)$.

Let $A(n):=k[X_1,\dots,X_n]/(X_1^p,\dots,X_n^p)~~(n\geq 1)$ be the truncated
polynomial ring.
Define
$$K''(n):=\{D\in W(n);~D.\omega_K\in A(n)\omega_K\},$$
where $n=2r+1$,~$\omega_K$ is the Cartan differetial form of type $K$ (cf.\cite[\S 4.1]{SF}).
By definition, the derived algebra $K(n):=[K''(n),K''(n)]$ is simple
and restricted, and we call it contact algebra (cf.  \cite[Page 174]{SF}).
According to \cite[Page 174, Theorem 5.5]{SF}, we know that $K(n)=K''(n)$ provided $n+3\neq 0\mod(p)$.
Following \cite[\S 4.5]{SF}, we endow the truncated polynomial ring $A(n)$
with the structure of a restricted Lie algebra; the corresponding Lie bracket $\langle,\rangle$ is called contact bracket.
More precisely, define the map (see \cite[Page 169]{SF})
\begin{equation}\label{Dkmap}
D_{K}:A(n)\rightarrow W(n).
\end{equation}
It turns out that the image of $D_K$ is exactly $K''(n)$.

First, we consider the case $n+3\neq 0 \mod(p)$, so that $K(n)=K''(n)$.
Let
$$\ft:=\oplus_{i=1}^rkx_ix_{r+i}\oplus k(1+x_n)\subseteq K(n).$$
According to \cite[page 1357]{BFS},
$\ft$ is a torus of maximal dimension in $K''(n)$.
Then $\ft$ is also a torus of maximal dimension of the centralizer $\mathfrak{c}:=C_{K''(n)}(1+x_n)$.
By \cite[(7.5.15)]{S1}, there is an isomorphism $\phi:\mathfrak{c}\rightarrow P(2r)$
of restricted Lie algebras such that $\phi(1+x_n)=1$,
where $P(2r)$ is the restricted Possion algebra with toral center (cf.\cite[Page 403]{S1}). 

Also, there is a short exact sequence of restricted Lie algebras:
\begin{equation}\label{ses of possion}
(0)\rightarrow k\rightarrow P(2r)\stackrel{D_H}{\rightarrow} H''(2r).
\end{equation}
Here $H''(2r)$ is the annihilator of the Hamiltonian form (see \cite[Page 163]{SF}).
In particular, $H''(2r)$ is a $p$-subalgebra of $W(2r)$.
Let $H'(2r)$ denote the image of $P(2r)$ under $D_H$
and the derived algebra $H'(2r)^{(1)}=H(2r)$ is the Hamiltonian algebra.
We obtain a short exact sequence of restricted Lie algebras:

\begin{equation}\label{ses of possion2}
(0)\rightarrow k\rightarrow P(2r)\stackrel{D_H}{\rightarrow}H'(2r)\rightarrow (0).
\end{equation}

\begin{Proposition}\label{rank H andP}
Keep the notations as above. Then the following statements hold:
\begin{enumerate}
\item[(1)] $\rk(H'(2r))=p^r-1$.
\item[(2)] $\rk(P(2r))=p^r$.
\item[(3)] $\rk(K''(n))=p^r$.
\end{enumerate}
\end{Proposition}
\begin{proof}
Note that $\mu(P(2r))=r+1$ (cf.\cite[Page 403]{S1}).
Lemma \ref{mu}(2) and (\ref{ses of possion2}) ensure that $\mu(H'(2r))=r$.
We also have $\dim_k H'(2r)=\dim_kP(2r)-1=p^{2r}-1$.
Since $\dim_k H(2r)=p^{2r}-2$ (see \cite[Page 166]{SF}) and $\mu(H(2r))=\mu(H'(2r))=r$ (see \cite[Corollary 7.5.9(2)]{S1}),
we have $H(2r)$ is a codimension $1$ restricted ideal of $H'(2r)$.
Now Corollary \ref{corollary} implies that $\rk(H'(2r))=\rk(H(2r))+1=p^r-1$ (Lemma \ref{D's result}(3)).
Furthermore, the exact sequence (\ref{ses of possion2}) and Proposition \ref{rk} imply
that $$\rk(P(2r))=r+1+\rk(H'(2r))-\mu(H'(2r))=p^r.$$
Note the centralizer $\mathfrak{c}:=C_{K''(n)}(1+x_n)$ is isomorphic to the restricted Possion algebra $P(2r)$
and $\rk(K''(n))=\rk(\mathfrak{c})$.
Now (3) follows from (2).
\end{proof}

\begin{Remark}
It should be notice that the following facts:
\begin{enumerate}
\item[(a)]$H(2r)\subseteq H'(2r)$, $\mu(H(2r))=\mu(H'(2r))=r$,
\item[(b)]$S(H(2r),\ft)\cong\GL_r(\FF_p)$.
\end{enumerate}
Hence, we can also use Corollary \ref{subalgebra rank} to compute 
the rank of $H'(2r)$.
\end{Remark}

\begin{Theorem}\label{rank of K}
Let $K(n)$ be the restricted contact algebra. Then
\begin{eqnarray}\rk(K(n))=
\begin{cases}
p^r, & n+3\neq 0\mod(p)\cr p^r-1,&n+3=0\mod(p)\end{cases}
\end{eqnarray}
\end{Theorem}
\begin{proof}
Thanks to Proposition \ref{rank H andP}, we only need consider the case $n+3=0\mod(p)$.
We know that $\dim_kK(n)=\dim_kK''(n)-1=\dim_kA(n)-1=p^{2r+1}-1$ (cf.\cite[\S 4.5]{SF}).
Since $K(n)$ is a restricted ideal of $K''(n)$ and the quotient is $p$-nilpotent,
our assertion immediately follows from Corollary \ref{corollary}.
\end{proof}

\begin{Corollary}\label{no self-centralizing tori of K}
The contact algebra $K(n)$ affords no self-centralizing torus except $p=n=3$.
\end{Corollary}
\begin{proof}
Suppose that there exists a self-centralizing torus.
Thanks to \cite[Corollary 3.8]{Fa04}, we know that every torus of $K(n)$
of maximal dimension is self-centralizing.
It follows from Lemma \ref{D's result}(4) and Theorem \ref{rank of K}
that 
$$r+1=p^r;~n+3\neq 0\mod(p)$$
or
$$r+1=p^r-1;~n+3= 0\mod(p).$$
By a direct computation, this holds only when $n=p=3$.
\end{proof}

\begin{Corollary}\label{dimensionof tori vari}
$$\dim\Tor(K(n))=p^n-p^r$$
\end{Corollary}
\begin{proof}
Thanks to \cite[(1.6)]{FV01} and \cite[(3.5)]{Fa04},
the dimension of the variety $\Tor(K(n))$ is $\dim_kK(n)-\rk(K(n))$.
Now our assertion follows from Theorem \ref{rank of K}.
\end{proof}

Let $(\fg,[p])$ be a restricted Lie algerba with connected automorphism group $G:=\Aut_p(\fg)^{\circ}$. A torus $\ft\in\Tor(\fg)$ is called generic if the orbit $G.\ft$ is a dense subset of the variety $\Tor(\fg)$.

\begin{Remarks}
(1). It is known that the contact algebra $K(n)$ affords no generic torus (\cite[Theorem 6.6]{BFS}).
In his doctoral dissertation \cite[Page 48]{Jensen}, 
the author give a direct proof of the fact.
It is not true since he used the incorrect fact that $K(n)$ admits a self-centralizing torus (see Corollary \ref{no self-centralizing tori of K}).
For convenience, we consider the case  $n+3\neq 0~\mod(p)$.
As before, we know that $K(n)=K''(n)$.
Let $G:=\Aut_p(K(n))^{\circ}$ be the connected automorphism group.
Accoring to \cite[Proposition 3.2]{LN} and Corollary \ref{dimensionof tori vari}, we know that
$$\dim G=\dim_k K(n)-(2r+1)=p^{2r+1}-(2r+1)>p^{2r+1}-p^r=\dim\Tor(K(n)).$$
Hence, we can not get the nonexistence of generic tori by directly comparing the dimensions of $G$ and $\Tor(K(n))$
(By assumption, we do not consider the case $r=1$ and $p=3$). 

(2). Suppose that $p>3$.
Let $(\fg,[p])$ be a non-classical restricted simple Lie algebra and $\ft\in\Tor(\fg)$ a torus of maximal dimension.
According to \cite[Theorem 5.3, 5.5 and Corollary 6.2]{BFS}, we have the following equality:
$$\dim_k\fg=\rk(\fg)+(p^{\mu(\fg)}-1)\dim_k\fg_\alpha,$$
where $\fg_\alpha$ is some root space.
By a direct check (see Lemma \ref{D's result},
Theorem \ref{rank of K} and \cite[Lemma 4.1]{Pre94} for the Melikian algebra), we obtain:
$$\dim_k\fg_\alpha=\max\{n\in\NN;~n(p^{\mu(\fg)}-1)<\dim_k\fg\}.$$ 
\end{Remarks}
\subsection*{Acknowledgements}
I would like to thank Professor Rolf Farnsteiner for his guidance and encouragement.

%%%%%%%%%%%%%%%%%%%%%%%%%%%%%%%%%%%%%%%%%%%%%%%%%%%%%%%%%%%%%%%%%%%%


\begin{thebibliography}{00}
\bibitem{BW} R. Block and R. Wilson,
\textit{Classification of the restricted simple Lie algebras}.
J. Algebra \textbf{114} (1988), 115--259.

\bibitem{BFS} J.-M. Bois, R. Farnsteiner and B. Shu,
\textit{Weyl groups for non-classical restricted Lie algebras and the Chevalley Restriction Theorem}.
Forum Math. \textbf{26} (2014), 1333--1379.

\bibitem{Chang} H.-J. Chang,
\textit{\"Uber Wittsche Lie-Ringe}.
Abh. Math. Sem. Hansischen Univ.\textbf{14} (1941). 151--184. 

\bibitem{D1} S. P. Demushkin,
\textit{Cartan subalgebras of simple Lie $p$-algebras $W_n$ and $S_n$ .}
Siberian Math. J. \textbf{11} (1970), 233--245.

\bibitem{D} S. P. Demushkin,
\textit{Cartan subalgebras of simple non-classical Lie-p-algebras.}
Math. USSR Izv. \textbf{6} (1972), 905--924.

\bibitem{Fa04} R. Farnsteiner,
\textit{Varieties of tori and Cartan subalgebras of restricted Lie algebras}.
Trans. Amer. Math. Soc. \textbf{356} (2004), 4181--4236.

\bibitem{FV01} R. Farnsteiner and D. Voigt,
\textit{Schemes of tori and the structure of tame restricted Lie algebras}.
J. London Math. Soc. \textbf{63} (2001), 553--570.

\bibitem{Hochschild1} G. Hoschschild,
\textit{Cohomology of restricted Lie algebras}.
Amer. J. Math. \textbf{76} (1954), 555--580. 

\bibitem{Jantzen} J. C. Jantzen,
\textit{Cartan matrices for restricted Lie algebras}.
Math. Z. \textbf{275} (2013).  569--594. 

\bibitem{Jensen} M. Jensen,
\textit{Invariant theory of restricted
Cartan type Lie algebras}.
PhD-Thesis, Aarhus University, 2015.

\bibitem{LN} Z. Lin and D. Nakano,
\textit{Algebraic group actions in the cohomology theory of Lie
algebras of Cartan type}.
J. Algebra \textbf{179} (1996), 852--888.

\bibitem{Pre} A. Premet,
\textit{On Cartan subalgebras of Lie $p$-algebras}.
Math. USSR Izv. \textbf{29} (1987), 145--157.

\bibitem{Pre94} A. Premet,
\textit{A generalization of Wilson's theorem on Cartan subalgebras of simple Lie algebras}.
J. Algebra \textbf{167} (1994), no. 3, 641--703.

\bibitem{S0} H. Strade,
\textit{Lie algebra representations of dimension $p-1$},
Proc. Amer. Math. Soc. \textbf{41} (1973), 419--424. 

\bibitem{S1} H. Strade,
\textit{Simple Lie algebras over fields of positive characterstic I:
Structure theory}.
de Gruyter Expositions in Mathematics, vol. \textbf{38},
Walter de Gruyter, Berlin/New York, 2004.

\bibitem{SF} H. Strade and R. Farnsteiner,
\textit{Modular Lie Algebras and their Representations}.
Pure and Applied Mathematics \textbf{116}. Marcel Dekker, 1988.

\bibitem{Winter} D. Winter,
\textit{On the toral structure of Lie $p$-algebras}.
Acta Math. \textbf{123} (1969), 69--81.

\end{thebibliography}
\end{document}